\documentclass[12pt]{article}
\usepackage{amsmath,amsxtra,amssymb,latexsym, amscd,amsthm}
\usepackage{tikz}
\usepackage{graphicx}
\newcommand{\R}{\mathbb{R}}
\newcommand{\F}{\mathcal F}
\renewcommand{\vec}[1]{\mathbf{#1}}
\newtheorem{theorem}{Theorem}[section]
\newtheorem{lemma}[theorem]{Lemma}
\newtheorem{corollary}[theorem]{Corollary}
\newtheorem{proposition}[theorem]{Proposition}

\begin{document}
\title{On the hollow enclosed by convex sets}

\author{
{Jen\H{o} Lehel}\\
{\small University of Louisville, and}\\
{\small Alfr\'ed R\'enyi Institute of Mathematics}\footnote{Since September 1, 2019 
the Alfr\'ed R\'enyi Institute of Mathematics  
does not belong to the Hungarian Academy of Sciences.}\\
{\small \texttt{lehel@louisville.edu}}\\
and\\
{G\'eza T\'oth}\thanks{Supported by the National 
Research, Development and Innovation Fund (TUDFO/51757/2019-ITM, 
Thematic Excellence Program) and 
National Research, Development and Innovation Office, NKFIH KKP-133864,
K-13152.}\\
{\small Alfr\'ed R\'enyi Institute of Mathematics, and}\\
{\small Budapest University of Technology and Economics, SZIT}\\
{\small \texttt{geza@renyi.hu}}}

\maketitle 
\abstract {\it For $n\leq d$, a family $\F=\{C_0,C_1,\ldots, C_n\}$ of compact convex sets in $\R^d$ is called an $n$-critical family provided any $n$ members of $\F$ have a non-empty intersection, but $\bigcap_{i=0}^n C_i=\varnothing$.    
If  $n=d$  then a lemma on the intersection of convex sets due to Klee implies that the $d+1$ members of the $d$-critical family  enclose a `hollow'  in $\R^d$, a bounded connected component of $\R^d\setminus\bigcup_{i=0}^n C_i.$   Here we prove that the closure of the convex hull of a hollow in $\R^d$ is a $d$-simplex. }\\

\noindent Besides the Helly-theorem on intervals 
 in $\R^1$ 
a less notable property is that two disjoint intervals can be separated by  a point, in other words, there is a `hollow' (an interval) between them, a gap, which cannot be bridged with two intervals having empty intersection. 
This separation or gap property, trivial as it is, helps characterize the intersection patterns of convex sets in $\R^1$ in terms of `interval graphs'. Actually, the gap property implies the foremost necessary condition that an interval graph must be chordal, namely, each cycle of length more than three has a chord (see \cite{chordal}). 
Just as Helly's theorem is established in $\R^d$, for every $d\geq 1$,  the separation or gap property has extensions to 
higher dimension. \\ 

\noindent  A family  of compact convex sets $C_0,C_1,\ldots, C_n\subset \R^d$ is called here an {\it $n$-critical} family
 if $\bigcap_{i\neq j} C_i\neq\varnothing$, for every $j=0,1,\ldots,n$, but $\bigcap_{i=0}^n C_i=\varnothing$.   
The denotation `critical'\footnote{The concept of criticality was introduced  in graph theory by T. Gallai \cite{Gallai}} becomes clear when in some finite 
family of sets with empty intersection we consider a `smallest' subset with the same property, a `critical subfamily'.

 Convexity and compactness in the definition of a critical family was chosen here with combinatorial geometry applications in mind (see \cite{JKLPT}). However,  in intersection or covering theorems of topology, when a finite or infinite family of sets appears, the compactness requirement of the members might be relaxed (see \cite{lasso}), and the condition $\bigcap_{i=0}^n C_i=\varnothing$ is usually replaced with its  contrapositive  that $\bigcup_{i=0}^n C_i$ is a convex set, which denies the hollow (see \cite{Klee}). Meanwhile, the primary condition that $\bigcap_{i\neq j} C_i\neq\varnothing$, for every $j=0,1,\ldots,n$, is unchanged  and  displays a topology variation of $n$-criticality in the different contexts. 
  
The role of $n$-critical families (or its variations) in Euclidean spaces was recognized by Klee \cite{Klee,Klee68}, Berge \cite{berge}, and Ghouila-Houri \cite{ghouila} in the study of intersection properties of convex sets.  
These properties are closely related to
fixed point theorems and minimax theorems  as explored by Fan \cite{Fan52}. As a result,  the intersection theorems and their applications were 
extended further in functional analysis and in topology by  Balaj \cite{Bal}, Ben-El-Mechaiekh \cite{Ben},  Fan \cite{Fan52,Fan84},  Horvath \cite{horvath} and others, by replacing the Euclidean space with general  topological vector spaces. All these investigations are originated in classical 
topology results such as the Sperner's lemma \cite{Sp}, and its generalizations starting with the Knaster, Mazurkievicz, Kuratowski-theorem \cite{KKM,lasso}. 

Observe that by Helly's theorem \cite{helly}, there is no $n$-critical family in $\R^d$ provided $n>d$.  A fundamental lemma due to 
Klee \cite{Klee} 
 implies that for $n=d$ there is a bounded  domain $D\subseteq\R^d\setminus \bigcup_{i=0}^d C_i$ 
 called here the {\it hollow} enclosed by the $d$-critical family in $\R^d$
  (Corollary \ref{hollow}).
 Section \ref{kleelemma} contains different proofs of Klee's fundamental  covering  lemma displaying its many faceted connections to combinatorial topology.  
In Section \ref{main} it is proved  that the closure of the convex hull of a hollow in $\R^d$ is a $d$-simplex (Theorem \ref{dencage}). An immediate corollary of  the hollow theorem, 
  related to an early result of Ghouila-Houri \cite{ghouila}, is formulated in
  Section \ref{applications} (Theorem \ref{cagethm}).  The note concludes with a separation property  of  $n$-critical families in $\R^d$, actually a corollary of a more general separation result by Klee \cite[Theorem 1]{Klee}, for the case $n<d$, when there is no hollow enclosed by the family (Theorem \ref{separ}).\\

\noindent Given a set $X\subset \R^d$,   the convex hull, the closure, and the boundary of $X$ is denoted by  Conv$(X)$, cl$(X)$, and $\partial X$, respectively.

\section{Klee's lemma}
\label{kleelemma}
 A basic lemma discovered by Klee \cite{Klee} and independently by  Berge \cite{berge}
captures a fundamental intersection property of $n$-critical families. We include here three proofs using different techniques and displaying a many faceted connections of the lemma to topology. The first purely geometry proof is  using the standard separation theorem of disjoint compact convex sets (c.f. \cite{Klee68}). The second proof  
was outlined by Berge \cite{berge} and applies a combinatorial topology result deduced from Sperner's  lemma  \cite{Sp}. The last proof uses the KKM lemma from fixed-point theory due to Knaster, Kuratowski, and Mazurkievicz \cite{KKM}.

 \begin{lemma} {\rm [Klee \cite{Klee}, Berge \cite{berge}]}.
 \label{lemma}
Let $C_0,C_1,\ldots, C_n\subset\R^d$ be compact convex sets such that 
$\bigcap\limits_{\substack {i=1\\ i\neq j}}^n C_i\neq\varnothing$, for every $j=0,1,\ldots,n$. If $\bigcup_{i=0}^n C_i$ is convex, then $\bigcap_{i=0}^n C_i\neq\varnothing $.
\end{lemma}
\begin{proof}
The proof is induction on $n$. The case $n=0$ is trivial; assume that $n\geq 1$ and 
the claim is true for $n$ convex sets. If $\bigcap_{i=0}^n C_i=\varnothing$, then  $C_n$ and $A=\bigcap_{i=0}^{n-1} C_i$ are disjoint compact convex sets, thus they 
can be strictly separated with a hyperplane $H$ such that $H\cap A=H\cap C_n=\varnothing$.
Let $C_i^\prime=H\cap C_i$, $0\leq i\leq n-1$.  

For every $j=0,\ldots,n-1$, the condition 
$\bigcap\limits_{\substack {i=1\\ i\neq j}}^n C_i= 
C_n\cap\left(\bigcap\limits_{\substack {i=1\\ i\neq j}}^{n-1} C_i\right)\neq\varnothing$ combined with  $H\cap C_n=\varnothing$ 
imply that $H\cap\left(\bigcap\limits_{\substack {i=1\\ i\neq j}}^{n-1} C_i\right)=\bigcap\limits_{\substack {i=1\\ i\neq j}}^{n-1} C_i^\prime\neq\varnothing$. Because 
$\bigcup_{i=0}^{n-1} C_i^\prime=(H\cap C_n)\cup\left(\bigcup_{i=0}^{n-1} H\cap C_i\right)=H\cap\left(\bigcup_{i=0}^{n} C_i\right)
$ 
is convex, we obtain 
by induction that $\bigcap\limits_{i=0}^{n -1}C_i^\prime=H\cap\left(\bigcap_{i=0}^{n -1}C_i\right)=H\cap A\neq\varnothing$, a contradiction. 
\end{proof}

\noindent {\it Second proof of Lemma \ref{lemma}}.
Let $a_j\in \bigcap_{i\neq j} C_i$,  for $j=0,1,\ldots,n$, and set 
$S=$ Conv$(\{a_0,\ldots, a_n\})$ for the convex hull of these $n+1$ points. If $S$ is not a simplex, then they span an affine  subspace of dimension $n-1$ or less, then by Helly's theorem the claim  $\bigcap_{i=0}^n C_i\neq\varnothing$ follows. We assume now that $S$ is an $n$-simplex. Since the facet $S^{(j)}\subset S$ opposite $a_j$ is included in $C_j$ and $\bigcup_{i=0}^n C_i$ is convex, we have $S\subseteq\bigcup_{i=0}^n C_i$. 

We take a simplicial subdivision of $S$ with arbitrary small mesh\footnote{\;mesh = the maximum diameter of the simplices of the subdivision}. 
A Sperner coloring\footnote{\;a vertex $v_i$ of the $n$-simplex $(v_0,\ldots,v_n)$ is  colored with $i$, $i=0,1,\ldots,n$, furthermore;\\ \indent\quad if $v\in$ Conv$(\{v_{i_0},v_{i_1},\ldots,v_{i_k}\})$  then the color of $v$ is any index from  $\{{i_0},{i_1},\ldots,{i_k}\}$} 
of the vertices of the subdivision is defined next. 
For a vertex $v$ of the subdivision let the color of $v$ be any index 
 $j\in\{{0},{1},\ldots,{n}\}$  such that $v\in C_{j-1}\setminus C_j$ (where $C_{-1}=C_n$). A color  $j$ exists for every $v\in S$, since otherwise, $v\in \bigcap_{i=0}^n C_i$, and the claim follows. Observe, if $j$ is the color of $v\in$ Conv$(\{a_{i_0},a_{i_1},\ldots,a_{i_k}\})$, then $j\in \{{i_0},{i_1},\ldots,{i_k}\}$  follows  by the convexity of $C_j$, and because  $v\notin C_j$.
  Then by Sperner's lemma, there is an $n$-simplex  
whose vertices are multicolored with $n+1$  different colors. 

By repeating the procedure with simplicial sudivisions of $S$ with mesh $\epsilon\searrow 0$, there is a convergent subsequence of the multicolored subdividing simplices approaching a point
$p\in S$. This limit point satisfies $p\in C_{j-1}$, for every $j=0,1,\ldots,n$, thus $\bigcap_{i=0}^n C_i\neq\varnothing$ follows.
\qed\\


The KKM lemma due to Knaster, Kuratowski, and Mazurkievicz \cite{KKM} is known as a remarkable intersection theorem for closed covers of a Euclidean simplex. 
Extending the Sperner lemma \cite{Sp}  
the KKM lemma was the starting point of further generalizations to topological vector spaces \cite{Ben,horvath,lasso}; these variations have been applied in mathematical fixed-point theory \cite{Fan84}.

A set-valued map $\Gamma$ of the points of an arbitrary  set $X\subset \R^d$ into sets of $\R^d$  is called a {\it KKM map on $X$} if for every finite subset
$N\subseteq X$, Conv$(N)\subseteq \bigcup_{x\in N} \Gamma(x)$. Ben-El-Mechaiekh \cite{Ben} proves a particular version of the KKM theorem stated as follows. 
\begin{theorem}
 \label{kkm}
If $\Gamma$ is a KKM map on $X\subset \R^d$ such that,  for every $x\in X$, $\Gamma(x)$
 is a non-empty closed convex subset of $\R^d$, then  
the family $\F=\{\Gamma(x)\}_{x\in X}$ has the finite intersection property, that is the intersection of the members of any finite subfamily of $\F$ is nonempty.\qed
\end{theorem}
For finite sets $X$ the claim in Theorem \ref{kkm} simply becomes 
$\bigcap_{x\in X} \Gamma(x)\neq\varnothing$.  
As observed by Ben-El-Mechaiekh \cite{Ben}, Klee's fundamental intersection theorem (Lemma \ref{lemma}) follows from
the finite version  of Theorem \ref{kkm}.\\

\noindent {\it Third proof of Lemma \ref{lemma}}.
Let   $a_j\in \bigcap_{i\neq j} C_i$,  for $j=0,1,\ldots,n$.
 Define the map $\Gamma(a_i)\mapsto C_{i-1}$, for $i=0,1,\dots,n$,  (where $ C_{-1}=C_n$). We verify that $\Gamma$ is a KKM map on
 $A=\{a_0,a_1,\ldots,a_n\}$; 
  let $N\subseteq A$.

For $N=A$,  because $A\subset \bigcup_{i=0}^n C_i$ and $C=\bigcup_{i=0}^n C_i$ is convex, we obtain
 Conv$(N)=$ Conv$(A)\subset C=\bigcup_{a_i\in N} \Gamma(a_i)$.
For $N\neq A$, let $j$ be an index such that $a_j\in N$, and $a_{j-1}\notin N$. Observe that $N\subset C_{j-1}$, and since $C_{j-1}$ is convex, 
 we obtain Conv$(N)\subset C_{j-1}=\Gamma(a_j)\subset \bigcup_{a_i\in N} \Gamma(a_i)$. 
 By Theorem \ref{kkm}, $ \bigcap_{a_i\in A} \Gamma(a_i)=\bigcap_{i=0}^n C_i\neq\varnothing$ follows.  \qed\\

\begin{corollary}
\label{hollow}
If $\{C_0,C_1,\ldots, C_d\}$ is a $d$-critical family in $\R^d$, then $\R^d\setminus \bigcup_{i=0}^d C_i$ has  a bounded connected component
$D$, that is every ray emanating from any point of $D$ intersects some $C_i$, $0\leq i\leq d$.
\end{corollary}

\begin{proof} Let $a_j\in \bigcap_{i\neq j} C_i$,  for $j=0,1,\ldots,d$. If 
 $E\subset \R^d$ is the affine space of dimension less than $d$, then the contradiction  
 $\bigcap_{i=0}^n C_i\neq \varnothing$ is obtained by Helly's theorem.
 Let $S=$ Conv$(\{a_0,\ldots, a_n\})$ be the $d$-simplex; notice that  
 each face of $S$ is contained in  $\bigcup_{i=0}^n C_i$.  
The compact convex sets $C^\prime_i=C_i\cap S$,  $i=0,1,\ldots,n$, form
a $d$-critical family, thus  by Lemma \ref{lemma} $\bigcup_{i=0}^d C_i^\prime\subset S$ is not convex, which means that  $S$ does not cover
$\bigcup_{i=0}^d C_i$. 
Let $p\in S\setminus\bigcup_{i=0}^d C_i$.
Because $\partial S\subseteq \bigcup_{i=0}^d C_i$, every ray emanating from  $p$ intersects $C_j$, for some $0\leq j\leq d$. 
\end{proof}

\section{The Hollow theorem} 
\label{main}
\begin{theorem} 
\label{dencage}
If ${\mathcal F}=\{C_0,\ldots,C_{d}\}$ is a $d-$critical family  in $\R^d$, then 
one of the connected components of $\R^d\setminus \bigcup_{i=0}^d C_i$ is a non-empty bounded region $D$, and the closure of {\em Conv$(D)$} is a $d$-simplex. 
\end{theorem}

\begin{proof}
The claim is true for $d=1$; let $d\geq 2$ and assume that the claim is true for $d-1$.
By Corollary \ref{hollow}, the hollow $D$ enclosed by  
${\mathcal F}$ exists. Furthermore, $D$ is an open set,  
 $\partial D\subseteq \partial C_0\cup\ldots\cup \partial C_{d}$,   and 
$D$ is contained in any $d$-simplex
$S$ with vertices in $ \bigcap_{h\neq j} C_h$, $j=0,\ldots,d$. Since $S$ is closed,  cl(Conv$(D)\subset S$.
 
For $j=0,\ldots, d$, let $p_j\in \bigcap_{h\neq j} C_h$ be a closest point of $\bigcap_{h\neq j} C_h$ to $C_j$.
We claim that $p_0,\ldots,p_{d}$ are unique points of $\partial D$. 
Assume that this claim is true, and  let
$S$ be the $d$-simplex with vertices $p_0,\ldots,p_{d}$. Because  cl(Conv$(D)$) is convex and
the vertices of $S$ belong to  $\partial D$, we have  
$S\subset$ cl(Conv$(D)$). On the other hand, 
we know
 cl(Conv$(D)$)$\subset S$, thus cl(Conv$(D)$) $= S$ follows.

1. We show that the simplex $S$ is unique.  Suppose that the points 
$a_1,a_2\in \bigcap_{h\neq d} C_h$  and $b_1, b_2 \in C_{d}$ are such that the minimum distance between $\bigcap_{h\neq d} C_h$ and 
$C_{d}$ is $m=|\overline{a_1b_1}|=|\overline{a_2b_2}|$.\footnote{\; $\overline{ab}$ is the line segment between points $a$ and $b$}
Let the  position vectors of  $a_i$ and $b_i$ be $\vec{a_i}$ and $\vec{b_i}$, respectively.  
By convexity,  $\vec{a}={1\over 2}(\vec{a_1}+\vec{a_2})\in\bigcap_{h\neq d} C_h$
and $\vec{b}={1\over 2}(\vec{b_1}+\vec{b_2})\in C_{d}$, hence $(\vec{a}-\vec{b})^2\geq m^2$. 
Using $(\vec{a_1}-\vec{b_1})^2=(\vec{a_2}-\vec{b_2})^2=m^2$ and setting $\gamma$ for the angle between $\vec{a_1}-\vec{b_1}$ and $\vec{a_2}-\vec{b_2}$ we obtain
$$
\begin{array}{lll}
2m^2\leq 2(\vec{a}-\vec{b})^2&=&{1\over 2}(\vec{a_1}-\vec{b_1}+\vec{a_2}-\vec{b_2})^2\\
\\
&=&{1\over 2}[(\vec{a_1}-\vec{b_1})^2+(\vec{a_2}-\vec{b_2})^2]+(\vec{a_1}-\vec{b_1})(\vec{a_2}-\vec{b_2})\\
\\
&=& m^2 +m^2\cos\gamma\leq 2m^2.
\end{array}
$$
This implies $\cos \gamma = 1$, that is  
$\overline{a_1b_1}\parallel\overline{a_2b_2}$, hence either $\overline{a_1b_1}=\overline{a_2b_2}$ or $(a_1,a_2,b_2,b_1)$ is a parallelogram.
 
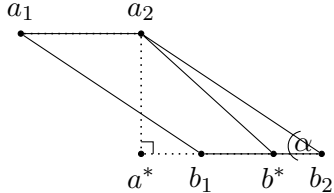
\begin{center} 
 \begin{tikzpicture}[scale=.4]
 \tikzstyle{P} = [circle, draw=black!, minimum width=2pt, inner sep=.5pt, fill=black]
\tikzstyle{txt}  = [circle, minimum width=1pt, draw=white, inner sep=0pt]
 \tikzstyle{O} = [circle, draw=black!, minimum width=3pt, inner sep=17pt]
   
 \node[P,label=above:$a_1$] (a1)at(-4,4){};
  \node[P,label=below:$a^*$] (a*)at(0,0){};
   \node[P,label=below:$b_1$] (b1)at(2,0){};
  \node[P,label=above:$a_2$] (a2)at(0,4){};
   \node[P,label=below:$b_2$] (b2)at(6,0){};
     \node[P,label=below:$b^*$] (b*)at(4.4,0){};
  \draw[dotted,line width=.6] (a*)--(b2) (a1)--(a2)--(a*) ;
  \draw(a1)--(b1)--(b2) (b*)--(a2)--(b2)(a1)--(a2);
  \draw (.4,0)--(.4,.4)--(0,.4);
  \node()at(5.4,0.3){$\alpha$};
  \draw(5.3,.75) arc (100:250:.5cm);
\end{tikzpicture}
\end{center}
Assume that $\overline{a_1b_1}$ and $\overline{a_2b_2}$ are distinct segments. If $(a_1,a_2,b_2,b_1)$ is not a rectangle, then set $\alpha=\angle a_2b_2b_1<\pi/2$. Let $a^*$ be the orthogonal projection of $a_2$ on the line
through $b_1,b_2$, and let $b^*\in \overline{b_1b_2}\cap \overline{a^*b_2}$. Then $b^*\in C_d$, and in the right triangle $(a_2,a^*,b_2)$ we have
$|\overline{a_2b^*}|<|\overline{a_2b_2}|=m$, a contradiction.
Thus we obtain that $(a_1,a_2,b_2,b_1)$ is a rectangle. 

The open ball of radius $m$ centered at $a_1$ is disjoint from $C_{d}$, hence
the hyperplane through $b_1$, $b_2$ and perpendicular to $\overline{a_1b_1}$ is a supporting hyperplane to $C_{d}$.  For every $j=0,\ldots,d-1$, select 
a point $c_j\in \bigcap_{h\neq j} C_h$. Apply Radon's theorem \cite{radon} on the
$(d+2)-$element set $R=\{a_1,a_2,c_0,\ldots,c_{d-1}\}$. Let $J_1\cup J_2=R$ be the Radon-partition, and let  $q\in$  Conv$(J_1)$ $\cap$ 
Conv$(J_2)$.
If $c_j\notin J_1$, then Conv$(J_1)\subset C_j$, and if $c_j\notin J_2$, then Conv$(J_2)\subset C_j$; therefore, $q\in$  Conv$(J_1)$ $\cap$ 
Conv$(J_2)\subset C_j$, for $j=0,\ldots,d-1$.
Thus we obtain that $q\in \bigcap_{j=0}^{d-1} C_j$, which implies $q\notin C_{d}$.   
Because Conv$(J_i\setminus\{a_1,a_2\})\subset C_d$ and $q\notin C_{d}$,  points $a_1, a_2$ are in distinct partition classes, say $a_i\in J_i$. Since $a_1\neq a_2$, we may assume that  $q\neq a_1$;  denote $m_0$ the distance of $q$ from $C_{d}$. Clearly, $m\leq m_0$.
 
  Because
 $\overline{a_1q}\subset$ Conv$(J_1)$ and
  Conv$(J_1\setminus\{a_1\})\subseteq C_{d}$,  
the line through $a_1$ and $q$ intersects $C_{d}$ at some point $c\in C_{d}$.
Our argument proceeds on the plane containing the triangle $(a_1,b_1,c)$. 
Let $q^\prime$ and $c^\prime$ be the points on the line 
through $c$ and $b_1$ such that $\overline{qq^\prime}\perp\overline{cb_1}$ and $\overline{a_1c^\prime}\perp\overline{cb_1}$ (see the figures).

 
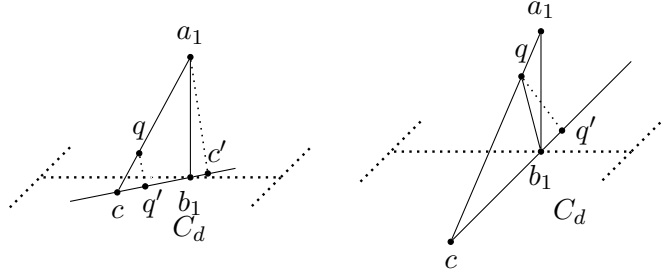
\begin{center} 
 \begin{tikzpicture}[scale=.4]
 \tikzstyle{P} = [circle, draw=black!, minimum width=2pt, inner sep=.5pt, fill=black]
\tikzstyle{txt}  = [circle, minimum width=1pt, draw=white, inner sep=0pt]
 \tikzstyle{O} = [circle, draw=black!, minimum width=3pt, inner sep=17pt]
\node()at(0,-3){};
\draw[dotted, line width=.9] (-1,1)--(-2,0)
(-3,-1)--(-2,0)--(6,0)--(7,1) (6,0)--(5,-1);
    
 \node[P,label=above:$a_1$] (a1)at(3,4){};
   \node[P,label=below:$b_1$] (b1)at(3,0){};
 
  \node[P,label=below:$c$] (c)at(.57,-.5){};
 \node[P,label=above:$q$] (q)at(1.3,.8){};
 \node[txt,label=below:$q^\prime$] ()at(1.75,0){};
  \node[P,label=right:$$] (qp)at(1.5,-.3){};
 \node[txt, label=above:$c^\prime$]()at(3.9,0.1){}; 
  \node[P,label=above:$$] (cp)at(3.58,0.13){}; 
\draw[dotted,line width=.6]  (q)--(qp) (cp)--(a1);

  \draw   (b1)--(a1)--(c)(-1,-.8)--(b1)--(4.5,.3);
  
\node[txt,label=right:$C_{d}$]() at (2,-1.7){};

\end{tikzpicture}
\hskip.5cm
\begin{tikzpicture}[scale=.4]
 \tikzstyle{P} = [circle, draw=black!, minimum width=2pt, inner sep=.5pt, fill=black]
\tikzstyle{txt}  = [circle, minimum width=1pt, draw=white, inner sep=0pt]
 \tikzstyle{O} = [circle, draw=black!, minimum width=3pt, inner sep=17pt]

 \draw[dotted, line width=.9] (-1,1)--(-2,0)
(-3,-1)--(-2,0)--(6,0)--(7,1) (6,0)--(5,-1);
    
 \node[P,label=above:$a_1$] (a1)at(3,4){};
   \node[P,label=below:$b_1$] (c1)at(3,0){};
   
  \node[P,label=below:$c$] (c)at(0,-3){};
 \node[P,label=above:$q$] (q)at(2.35,2.5){};
   \node[P,label=right:$q^\prime$] (qp)at(3.7,.7){};
   
\draw[dotted,line width=.6]  (q)--(qp) ;

  \draw   (q)--(b1)--(a1)--(c)--(c1)--(6,3);
  
\node[txt,label=right:$C_{d}$]() at (3,-2){};
\end{tikzpicture}
\end{center}

If $q^\prime\in \overline{cb_1}$  then by convexity,
$q^\prime\in C_{d}$. This implies that 
$m_0\leq|\overline{qq^\prime}|
< |\overline{a_1c^\prime}|
\leq |\overline{a_1b_1}|=m\leq m_0,$ a contradiction (see the figure on the left). 
If $b_1\in \overline{cq^\prime}$  then we have  $\angle{b_1qa_1}=\pi - \angle{cqb_1}\geq \pi-\angle{cqq^\prime}>\pi/2$ (see on the right). Therefore,  
 $m_0\leq|\overline{qb_1}|< 
 |\overline{a_1b_1}|=m\leq m_0,
$ a contradiction.

We conclude that $a_1=a_2$, thus $p_{d}$ is uniquely determined as the closest point in $\bigcap_{h\neq d} C_h$ to $C_{d}$. Similarly, each point $p_i\in\bigcup_{h\neq i} C_h$
closest to $C_i$, $i=0,\ldots,d-1$, is uniquely determined. Furthermore, because $\bigcap_{i=0}^{d} C_i=\varnothing$, $S=(p_0,\ldots,p_{d})$  is a $d-$simplex. 

2. Next we show that $p_{d}\in \partial H$. Let $b\in \partial C_{d}$ be the closest point in $C_{d}$ to $p_{d}\in \bigcap_{h\neq d}C_h$. For $i=0,1,\ldots,d-1$, let $a_i\in \partial C_{d}\cap  \left(\bigcap_{h\neq i}C_h\right)$. We translate the point $b$ to $p_d$, and assume that the same translation takes the points $a_0,\ldots,a_{d-1}$ into $a_0^\prime,\ldots a_{d-1}^\prime$, respectively. 
Define $B=\partial C_d\cap$ Conv$(\{b,a_0,\ldots,a_{d-1}\}\cup\{p_{d},a_0^\prime,\ldots a_{d-1}^\prime\})$, 
  and let $B^\prime$ be the translation of $B$ sending $b$ into $p_{d}$. 
Observe that $ \bigcap_{h\neq d}C_h$ has no point in the interior of 
$Q=$ Conv$(B\cup B^\prime)$.

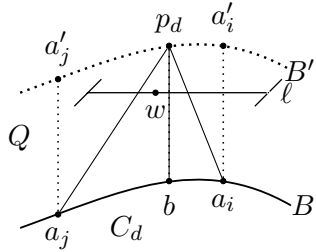
 \begin{center} 
\begin{tikzpicture}[scale=.6]
 \tikzstyle{P} = [circle, draw=black!, minimum width=2pt, inner sep=.5pt, fill=black]
\tikzstyle{txt}  = [circle, minimum width=1pt, draw=white, inner sep=0pt]
 \tikzstyle{O} = [circle, draw=black!, minimum width=3pt, inner sep=17pt]
 
  \draw[line width=.7] (0,0) to[out=20,in=150] (6,.5); 
    \draw[dotted,line width=.9] (0,3) to[out=20,in=150] (6,3.5); 
 \node[P,label=above:$a_i^\prime$] (bp)at(4.5,4.04){}; 
  \node[P,label=below:$a_i$] (b)at(4.5,1.04){};
  \node[P,label=below:$b$] (c)at(3.3,1.04){};
 \node[P,label=above:$p_{d}$] (p*)at(3.3,4.04){};
 
 \node[P,label=below:$a_j$] (a)at(.84,.3){};
 \node[P,label=above:$a_j^\prime$] (ap)at(.84,3.3){};
 
 \node[P,label=below:$w$] (w)at(3,3){};
 \draw[line width=.5]  (1.2,2.7)-- (1.5,3)(1.8,3.3)-- (1.5,3)--(5.5,3)--(5.8,3.3) (5.5,3)--(5.2,2.7); 
 \node[txt,label=right:$\ell$](C) at (5.5,3){};
 
\node[txt,label=left:$C_{d}$](C) at (3,0){};
\node[txt,label=left:$Q$](Q) at (0.5,2){};
\node[txt,label=left:$B^\prime$](Q) at (6.8,3.5){};
\node[txt,label=left:$B$](Q) at (6.8,.5){};

\draw[dotted,line width=.6] (p*)--(c) (a)--(ap) (b)--(bp);
\draw (p*)--(c) ;
\draw (b)--(p*)--(a);

\end{tikzpicture}

\end{center}

Now we take a hyperplane $\ell$ strictly separating $p_{d}$ from $B$ and sufficiently close to $p_{d}$. The intersection of 
 $C=$ Conv$(\{p_{d},a_0,\ldots,a_{d-1}\})$  with  $\ell$ is inside the interior of $Q$; let  $L=\ell\cap C$. 
The convex sets $C_i^\prime=C_i\cap L$, $i=0,1,\ldots,d-1$, form a
$(d-1)$-critical family ${\mathcal F}^\prime$ in the hyperplane $\ell$. By induction, the hollow enclosed by ${\mathcal F}^\prime$ in  $\ell$ 
contains a point $w\in L\setminus\left(\bigcup_{i=0}^{d-1}C_i\right)$.
The simplex Conv$(\{p_{d},a_0,\ldots,a_{d-1}\})$ contains 
the hollow $H$ enclosed by ${\mathcal F}$ in $\R^d$, which implies that $w\in H$.

Because $\ell$ can be taken arbitrarily close to $p_{d}$, 
the point $w\in H$ becomes arbitrarily close to $p_d$. Thus we obtain $p_{d}\in \partial H$, and similarly, $p_i\in \partial H$,  $0\leq i\leq d-1$. 
Therefore, cl(Conv$(H)$)$=$Conv$(\{p_0,p_1,\ldots,p_{d}\})$.
\end{proof}

\section {Conclusion}
 \label{applications}

Given a $d-$critical family $\F=\{C_0,\ldots, C_{d}\}$ in $\R^d$,  a {\it cage}  is defined as a closed set containing $d+1$ {\it base points}, 
$a_i\in \bigcap_{h\neq i} C_h$, $0\leq i\leq d$. 
A  convex cage $M$ carried by $\F$ contains the hollow $D\subset \R^d\setminus \bigcup_{i=0}^d C_i$ enclosed by the family, because $D$ is included in the convex hull of the base points of $M$.  The generalization of Berge's theorem \cite{berge} due to Ghouila-Houri  \cite{ghouila} implies the following property of a convex cage (as a special case).

 \begin{proposition} 
 \label{GHtheorem}
 Let ${\mathcal F}=\{C_0,\ldots, C_{d}\}$ be a $d-$critical family in $\R^d$,
 and  let $F$ be a closed set containing the hollow $D$ enclosed by $\F$. 
 If $M$ is a convex cage carried by ${\mathcal F}$, then $F\cap M$ is also a cage. \qed
 \end{proposition}
    When applying Proposition \ref{GHtheorem} with $F=$ cl(Conv($D))$, 
    then the $d+1$ base points of  the cage $F\cap M$ may depend on the choice of $M$.  Theorem \ref{dencage} implies that this is not the case, 
    Proposition \ref{GHtheorem} is true in a stronger form, namely, there is a unique convex cage minimal by inclusion, the $d$-simplex cl(Conv($D))$.
    
    \begin{theorem}
    \label{cagethm}
     Let ${\mathcal F}=\{C_0,\ldots, C_{d}\}$ be a $d-$critical family in $\R^d$.   Then there exist $d+1$ base points,  
 which belong to  every convex cage $M$ carried by $\F$.\qed
   \end{theorem}

 If $n<d$ then there is no hollow enclosed by the members of an $n$-critical family in $\R^d$. In particular, the two compact convex members of a $1$-critical family do not enclose a hollow in $\R^2$; nevertheless, since they are disjoint, they can be strictly separated by a line. A result due to 
Klee \cite[Theorem 1]{Klee} 
extends this separation property in $\R^d$ for any $n$-sets.\footnote{the concept of an {\it $n$-set}  is a variation of $n$-criticality used by Klee \cite{Klee}} 
Klee's separation theorem
has  an immediate corollary for $n$-critical families below;
a simple proof (extending easily the induction proof of Lemma \ref{lemma} given above) is due to Breen \cite{Breen}.

\begin{theorem} {\rm (Breen \cite{Breen})}. 
\label{separ} 
For $1\leq n\leq d$, let $\{C_0,C_1,\ldots, C_n\}$ be an $n$-critical family in $\R^d$, and let $a_i\in\bigcap_{h\neq i}C_h$, $0\leq i\leq n$. Then  
 in $\R^d$ there
are two affine subspaces, $W$ of dimension $n$ and $V$ of dimension $d-n$ (called {a stabbing affine subspace}), meeting in a single point $p$ and such that
\begin{itemize}
\item[{\em (a)}] $V\cap C_i=\varnothing$ and $a_i\in W$, for every $0\leq i\leq n$, and
\item[{\em (b)}] the set $W\bigcap \left(\bigcup_{i=0}^n C_i\right)$ surrounds\footnote{\,$Q$ surrounds $P$ in $A$ 
if $A\setminus Q$ has a connected component which is bounded and contains $P$} $\{p\}$ in $W$.
\qed
\end{itemize}
\end{theorem}

The special version of Theorems \ref{dencage} and \ref{cagethm}  for $d=2$ was originally developed and applied  by Jobson et al.  \cite[Lemma 1]{JKLPT} in the study  of an extremal problem involving forbidden planar convex hypergraphs. It is worth noting that the characterization of $d$-dimensional convex hypergraphs\footnote{
vertices 
 are convex sets in $\R^d$, and $d+1$ vertices form a hyperedge if and only if they\\
 \indent\quad  
 have nonempty intersection}
 is not known for $d\geq 2$. For $d=1$ the convex hypergraphs are called  interval graphs; and as it is well known, their characterization was done by Lekkerkerker and Boland \cite{LB} in terms of forbidden obstructions, and by Gilmore and Hoffman \cite{interval} using the ordering and the separation property of the real line.

\smallskip

Having Theorem \ref{separ}, one could try to generalize
the Hollow Theorem (Theorem \ref{dencage}), 
that is, for an $n$-critical family $\{C_0,C_1,\ldots, C_n\}$ in $\R^d$,
one might ask for some kind of `geombinatorial' description 
of the set of all stabbing  
$(d-n)$-dimensional affine spaces $V$.
At this point we do not even have a reasonable conjecture.

\end{document}